\documentclass[10pt]{amsart}
\usepackage{graphicx}
\usepackage{tikz}
\usepackage{tikz-cd}
\usepackage{float}
\usepackage{tabularx}
\usepackage{mathtools}
\usepackage{ amsbsy} 
\usepackage{amsmath}
\usepackage{amssymb}
\usepackage[utf8]{inputenc}
\usepackage{graphicx}%
\usepackage{multirow}%
\usepackage{amsmath,amssymb,amsfonts}%
\usepackage{amsthm}
\usepackage{mathrsfs}%
\usepackage[title]{appendix}%
\usepackage{xcolor}%
\usepackage{textcomp}%
\usepackage{manyfoot}%
\usepackage{booktabs}%
\usepackage{listings}
\usepackage{amssymb}
\usepackage{amsfonts}
\usepackage{amsmath}
\usepackage{epsfig}
\usepackage{indentfirst}
\usepackage[usenames,dvipsnames]{pstricks}
\usepackage{pst-grad}
\usepackage{pst-plot}
\usepackage[margin=1.0in]{geometry}
\usepackage{graphicx}
\usepackage{tikz}
\usepackage{tikz-cd}
\usepackage{float}
\usepackage{tabularx}
\usepackage{mathtools}

 \newtheorem{theorem}{Theorem}[section]
 \newtheorem{corollary}[theorem]{Corollary}
 \newtheorem{lemma}[theorem]{Lemma}
 \newtheorem{proposition}[theorem]{Proposition}
 \theoremstyle{definition}
 
 \theoremstyle{remark}
 \newtheorem{remark}[theorem]{Remark}
 
 \numberwithin{equation}{section}

\DeclareMathOperator {\Hom}{Hom}

\DeclareMathOperator {\ind}{ind}
\DeclareMathOperator {\Res}{Res}

\DeclareMathOperator {\Ext}{Ext}

\DeclareMathOperator {\Gl}{GL}

\newcommand{\Z}{\mathbb{Z}}

\newcommand{\C}{\mathbb{C}}

\newcommand{\A}{\mathcal{A}}
\newcommand{\B}{\mathcal{B}}

\newcommand{\Sg}{\mathfrak{S}}
\newcommand{\h}{\mathcal{H}}
\newcommand{\bdel}{\pmb{\delta}}

\newcommand{\bolam}{\pmb{\lambda}}
\newcommand{\bomu}{\pmb{\mu}}

\begin{document}

\renewcommand{\baselinestretch}{1.0}
\title[Comparing cohomology via exact split pairs in diagram algebras]
 {Comparing cohomology via exact split pairs in diagram algebras}

\author{Sulakhana Chowdhury}\address{Indian Institute of Science Education and Research Thiruvananthapuram, Thiruvananthapuram, \newline Email: sulakhana17@iisertvm.ac.in}
\author{Geetha Thangavelu}\address{Indian Institute of Science Education and Research Thiruvananthapuram, Thiruvananthapuram, \newline Email: tgeetha@iisertvm.ac.in }


\begin{abstract}
In this article, we compare the cohomology between the categories of modules of the diagram algebras and the categories of modules of its input algebras. Our main result establishes a sufficient condition for exact split pairs between these two categories, analogous to a work by Diracca and Koenig \cite{DK}. To be precise, we prove the existence of the exact split pairs in $A$-Brauer algebras, cyclotomic Brauer algebras, and walled Brauer algebras with their respective input algebras. 

\end{abstract}

%%% ----------------------------------------------------------------------
\maketitle
%%% ----------------------------------------------------------------------
%\tableofcontents
\section{Introduction}
Representation theory stands as one of the cornerstones among the disciplines in mathematics, with its primary objectives revolving around the classification and description of irreducible (also known as simple) representations. In the context of ring homomorphisms, directly comparing the cohomology of two rings is often a challenging task. However, in \cite{DK}, a method is introduced by Diracca and Koenig to overcome this difficulty. This approach combines the aspects of subring and quotient ring situations, allowing for cohomology comparison without the need for strong assumptions typically used in either of these scenarios. This technique can be used for the non-vanishing of cohomology in specific situations. The concept of exact split pair of functors provides an interesting tool to compare the cohomology between rings under particular conditions.

Diagram algebras play a significant role in the Schur-Weyl duality, a connection between representations of diagram algebras and algebraic groups. The classical Schur-Weyl duality, establishes a relationship between the polynomial representations of general linear groups $\Gl_n$ over the complex numbers $\C$ and representations of the symmetric groups $\Sg_r$, for details see \cite{JAM,Schur1927}.  Later on, Brauer examined an analogous problem for the symplectic and the orthogonal groups. The results on the Brauer algebras have been extended to walled Brauer algebras, Brauer algebras of type C,  cyclotomic Brauer algebras, and $A$-Brauer algebras. The walled Brauer algebras were introduced as centralizer algebras of $\Gl_n$ acting on the mixed tensor space in \cite{BCHL}, and are subalgebras of the Brauer algebras. The cyclotomic Brauer algebras were studied by Rui, Yu in \cite{RY1}, Bowman, Cox, and Visscher in \cite{BCV}, and $A$-Brauer algebras were studied in \cite{GG}. The primary focus of this article is to study an exact split pair for these algebras.

In this article, we use the framework of cellular algebras introduced by Koenig and Xi in \cite{KX,KX3}, where every cellular algebra can be written as an iterated inflations involving smaller cellular algebras. Using this framework Hartmann, Henke, Koenig, and Paget studied the cellularly stratified algebras and their split pairs in [\cite{HHKP}, Section 3]. $A$-Brauer algebras are known to be cyclic cellular but the cellularly stratified structure is yet to be known. Therefore, we used results from \cite{DK} to obtain an exact split pair for $A$-Brauer algebras, cyclotomic Brauer algebras and the walled Brauer algebras. 

The primary objective is to determine the sufficient condition for the existence of an exact split pair for the above mentioned class of diagram algebras and we use these to compare $\Hom,\Ext$ and study the associated numerical invariants, including the global dimensions.

This article is organized as follows. In Section \ref{Preliminaries}, we review definitions and properties of an exact split pair. In Section \ref{$A$-Brauer algebra}, we show an iterated inflation for $A$-Brauer algebras, which allows us to prove the existence of an exact split pair for the $A$-Brauer algebras. In Section \ref{Cyclotomic Brauer algebra} and Section \ref{Walled Brauer algebras}, we prove our main result for the cyclotomic Brauer algebras and the walled Brauer algebras.

\section{Preliminaries} \label{Preliminaries}

In this section, we recall some definitions from \cite{DK}, which will be used in this article.
 
%\begin{definition}
Let $\A$ and $\B$ be two additive categories. A pair $(F,G)$ of two additive functors $F: \A \longrightarrow \B$ and $G: \B \longrightarrow \A$ is a \textit{split pair of functors} if the composition $F \circ G$ is an auto-equivalence of $\B$. If $\A$ and $\B$ are equipped with exact structures, and $F$ and $G$ are exact functors with respect to these exact structures, then the split pair is called an \textit{exact split pair of functors}. 
%\end{definition}

%\begin{definition}%(Split quotient)
Let $A$ and $B$ be two rings. We call $B$ a \textit{split quotient} of $A$ if there is an embedding $\epsilon$ sending the unit of $B$ to the unit of $A$ and there exists a surjective homomorphism $\pi: A \twoheadrightarrow B$ such that $\pi \circ \epsilon=1_{B}$. Let
$\textbf{\text{mod}-}A$ and $\textbf{\text{mod}-}B$ be the categories of finitely generated right $A$-modules and right $B$-modules, respectively.
%\begin{center}
%\begin{tikzcd}
%B \arrow[r, hook, "\epsilon"] \arrow[swap,""{name=G}]{dr} {\text{id}}
%& A \arrow[d,twoheadrightarrow,"\pi"]\\
%& B\arrow[from=1-2, to=G, pos=.4, phantom, "\circlearrowleft"]
%\end{tikzcd}
%\end{center}
The homomorphism $\pi$ and $\epsilon$ induce two exact functors $F$ and $G$ between the exact categories $\textbf{\text{mod}-}A$ and $\textbf{\text{mod}-}B$ such that $F\circ G$ is the identity on $\textbf{\text{mod}-}B$, where $F= -\otimes_A A_B$ and $G=-\otimes_B B_A$.
% \begin{center}
%\begin{tikzcd}
%\B \arrow[r, hook, "G"] \arrow[dashed,swap, ""{name=G}]{dr}{F\circ G}
%& \A \arrow[d,twoheadrightarrow,"F"]\\
%& \B\arrow[from=1-2, to=G, pos=.4, phantom, "\circlearrowleft "]
%\end{tikzcd}
%\end{center}
%\end{definition}

%\begin{definition}%(Corner rings) 
Let $e$ be an idempotent of a ring $A$. Let $B$ be the centralizer subring $eAe$. Consider $eA$ and $Ae$ to be the $B$-$A$ and $A$-$B$ bimodules, respectively. Assume $Ae$ to be $eAe$-projective. Then the functors 
\begin{align*}
F: \textbf{\text{mod}-}A &\longrightarrow \textbf{\text{mod}-}B      &G:\textbf{\text{mod}-}B &\longrightarrow \textbf{\text{mod}-}A\\
N &\longmapsto N \otimes_A Ae  &M &\longmapsto M\otimes_B eA
\end{align*}
form an exact split pair of functors. We call the centralizer subrings $eAe$ as \textit{corner rings}.
%\end{definition}

%(Corner split quotient)
Let $e$ be an idempotent in $A$, and $B$ a split quotient of $eAe$, viewed as a subring of $eAe$. Then we call $B$ a \textit{corner split quotient} of $A$ with respect to $e$,  if there is an $eAe$-$A$ bimodule $S$, which is projective as left $B$-module via the embedding of $B$ into $eAe$, and also satisfies the condition $Se\cong B$ as right $B$-modules. 

%\s{write a homological aspects of these defn} \s{mention the theorem regarding split pair}

To prove the sufficient condition for the existence of an exact split pair, we will frequently
use the following result.

\begin{theorem}[\cite{DK}, Lemma 3.2] \label{general corner split quotient}
Let $A$ and $B$ be any arbitrary rings and $S$ a left $B$-module. If $B$ is a corner split quotient of $A$. Then the functors $F = - \otimes_A Ae$ and $G =- \otimes_B B \otimes_{eAe} S$ form an exact split pair.
\end{theorem}

\section{The $A$-Brauer algebra} \label{$A$-Brauer algebra}

Let $R$ be a commutative unital ring, with $\delta \in R$. Let $A$ be a unitary associative algebra over $R$, not necessarily commutative. Assume that $A$ possesses an $R$-linear algebra involution denoted by $*$, along with a $*$-invariant $R$-valued trace, denoted as $\mathrm{tr}$, where $\mathrm{tr}(1)= \delta$. An $A$-Brauer diagram is a Brauer diagram with $A$-labeled edges and has an orientation of edges. We denote the $A$-Brauer algebra by $D_n(A)$. Note that the algebra $D_n(A)$ is generated by $s_{i}$, $e_{i}$, and $h_j^{a}$ where $a\in A, 1 \leq i \leq n-1$, and $1\leq j \leq n$ (see Figure \ref{generators of the A-Brauer algebra}).

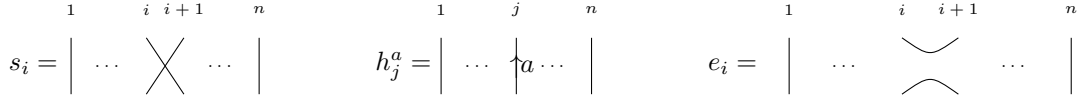
\begin{figure}[H]
\begin{tikzpicture}[x=0.5cm,y=0.75cm]
\draw[-] (0,0)-- (0,-1);
\draw[-] (2,0)-- (3,-1);
\draw[-] (3,0)-- (2,-1);
\draw[-] (5,0)-- (5,-1);
\node at (-1,-0.5) {$s_{i}=$};
\node at (0,0.5) {\tiny $1$};
\node at (1,-0.5) {\tiny $\cdots$};
\node at (2,0.5) {\tiny $i$};
\node at (3,0.5) {\tiny $i+1$};
\node at (4,-0.5) {\tiny $\cdots$};
\node at (5,0.5) {\tiny $n$};
\end{tikzpicture}
\hspace{1cm}
\begin{tikzpicture}[x=0.5cm,y=0.75cm]
\draw[-] (0,0)-- (0,-1);
\draw[-] (2,0)-- (2,-1);
\draw[-] (4,0)-- (4,-1);
\node at (-1,-0.5) {$h_{j}^{a}=$};
\node at (0,0.5) {\tiny $1$};
\node at (1,-0.5) {\tiny $\cdots$};
\node at (2,0.5) {\tiny $j$};
\node at (2,-0.5) { $\uparrow$};
\node at (2.3,-0.5) { $a$};
\node at (3,-0.5) {\tiny $\cdots$};
\node at (4,0.5) {\tiny $n$};
\end{tikzpicture}
\hspace{1cm}
\begin{tikzpicture}[x=0.75cm,y=0.75cm]
\draw[-] (0,0)-- (0,-1);
%\draw[-] (1,0)-- (1,-1);
\draw[-] (2,0).. controls(2.5,-0.35)..(3,0);
\draw[-] (2,-1).. controls(2.5,-0.65)..(3,-1);
\draw[-] (5,0)-- (5,-1);
\node at (-1,-0.5) {$e_{i}=$};
\node at (0,0.5) {\tiny $1$};
\node at (1,-0.5) {\tiny $\cdots$};
\node at (2,0.5) {\tiny $i$};
\node at (3,0.5) {\tiny $i+1$};
\node at (4,-0.5) {\tiny $\cdots$};
\node at (5,0.5) {\tiny $n$};
\end{tikzpicture}
\caption{Generators of the $A$-Brauer algebra.}\label{generators of the A-Brauer algebra}
\end{figure}

Let $X$ and $Y$ be two $A$-Brauer diagrams. The product $X\cdot Y$ is defined by placing $X$ on top of $Y$ and concatenating the bottom row of $X$ with the top row of $Y$ to form a new diagram $Z$. The product rule is similar to that of Brauer diagrams. Fix an orientation of labels in the concatenated diagram and apply $*$ on the label to reverse the orientation if needed. Each edge in the product is labeled by the product of the labels of the component edges.  For each loop, multiply the diagram by the $R$-valued trace of its label. Then, $X\cdot Y= \alpha Z$, where the scalar $\alpha$ is the product of $R$-valued trace of all closed loops. For example, see [\cite{GG}, Figure 5.3]. For more details, see [\cite{GG}, Section 5].

%\s{check}Now, define a functional $\mathrm{Tr}$ on $A^{\otimes n} \times D_n$  as follows: let $a=a_1 \otimes\cdots \otimes a_n$ be a simple tensor in $A^{\otimes n}$. For a $d \in D_n$, let $o =(i, d(i), d^2(i), \cdots, d^k(i))$ be an orbit for the action of powers of $d$ on $\{1, 2, \cdots, n\}$. Let $o(a) = \mathrm{tr}(a_{d^k(i)} \cdots a_{d(i)} a_i)$. Deﬁne $\mathrm{Tr}(da)= \prod o(a)$, where the product is over the orbits of powers of $d$. Thus, $\mathrm{Tr}$ is a $*$-invariant trace on $A^{\otimes n} \times D_n$ such that $\mathrm{Tr}(1)= \delta_0$.

Let $A \wr \Sg_n$ be the wreath product algebra, a subalgebra of $D_n(A)$ generated by $s_{i}$ and $h_j^{a}$ for $1 \leq i \leq n-1,a \in A$, and $1 \leq j \leq n$.  A cellular algebra is said to be \textit{cyclic cellular} if every cell module of $A$ is cyclic [\cite{GG}, Section 2.4.3]. Let $(\Gamma, \geq)$ be a finite poset. Let $A$ be a cyclic cellular algebra with cell datum $(A, *, \Gamma, \geq, \mathcal{T}, \mathcal{C})$, described as in [\cite{GG}, Lemma 2.5]. If the algebra $A$ is cyclic cellular, then $A\wr \Sg_n$ is cyclic cellular  with cell datum $(A\wr \Sg_n, *, \Lambda_n^{\Gamma}, \unrhd_{\Gamma}, \mathcal{M}, \mathcal{B} )$, see [\cite{GG}, Theorem 4.1].

When $\delta $ is invertible in $R$, the idempotent is given by Figure \ref{Idempotent of the A-Brauer algebra when delta is nonzero}. When $\delta =0$ and $n$ is odd, the idempotent is defined as in Figure \ref{Idempotent of the A-Brauer algebra when delta 0}. When $\delta = 0$ and $n$ is even, using the same idempotent as in Figure \ref{Idempotent of the A-Brauer algebra when delta is nonzero} results in the square of the two-sided ideal $J$ being zero, where $J$ is the two-sided ideal of $D_n(A)$ containing all $A$-Brauer diagrams with only horizontal edges. Hence, we shall not consider the case $\delta = 0$ and $n$ is even.

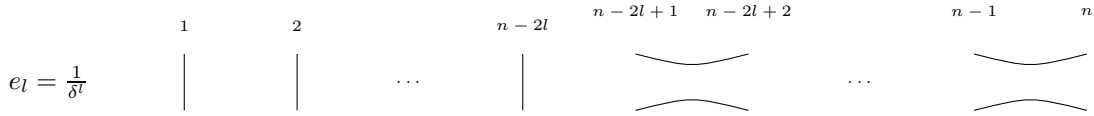
\begin{figure}[H]
\begin{tikzpicture}[x=1.5cm,y=0.75cm]
\draw[-] (1,0)-- (1,-1);
\draw[-] (2,0)-- (2,-1);
\draw[-] (4,0)-- (4,-1);
\draw[-] (5,0).. controls(5.5,-0.25)..(6,0);
\draw[-] (8,0).. controls(8.5,-0.25)..(9,0);
\draw[-] (5,-1).. controls(5.5,-0.75)..(6,-1);
\draw[-] (8,-1).. controls(8.5,-0.75)..(9,-1);

%\draw[\mid] (3.5,0.5)-- (3.5,-1.5);
\node at (-0.2,-0.5) {$e_l=\frac{1}{\delta^l}$};
\node at (1,0.5) {\tiny $1$};
\node at (2,0.5) {\tiny $2$};
\node at (3,-0.5) {\tiny $\cdots $};
\node at (4,0.5) {\tiny $n-2l$};
\node at (5,0.75) {\tiny $n-2l+1$};
%\node at (5.5,0.25) {\tiny $a_l$};
\node at (6,0.75) {\tiny $n-2l+2$};
\node at (7,-0.5) {\tiny $\cdots$};
\node at (8,0.75) {\tiny $n-1$};
%\node at (8.5,0.25) {\tiny $a_1$};
\node at (9,0.75) {\tiny $n$};
\end{tikzpicture}
\caption{Idempotent of the $A$-Brauer algebra for an invertible $\delta$.}\label{Idempotent of the A-Brauer algebra when delta is nonzero}
\end{figure}

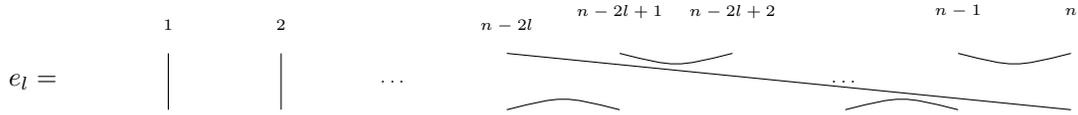
\begin{figure}[H]
\begin{tikzpicture}[x=1.5cm,y=0.75cm]
\draw[-] (1,0)-- (1,-1);
\draw[-] (2,0)-- (2,-1);
\draw[-] (4,0)-- (9,-1);
\draw[-] (5,0).. controls(5.5,-0.25)..(6,0);
\draw[-] (8,0).. controls(8.5,-0.25)..(9,0);
\draw[-] (4,-1).. controls(4.5,-0.75)..(5,-1);
\draw[-] (7,-1).. controls(7.5,-0.75)..(8,-1);

%\draw[\mid] (3.5,0.5)-- (3.5,-1.5);
\node at (-0.2,-0.5) {$e_l=$};
\node at (1,0.5) {\tiny $1$};
\node at (2,0.5) {\tiny $2$};
\node at (3,-0.5) {\tiny $\cdots $};
\node at (4,0.5) {\tiny $n-2l$};
\node at (5,0.75) {\tiny $n-2l+1$};
%\node at (5.5,0.25) {\tiny $a_l$};
\node at (6,0.75) {\tiny $n-2l+2$};
\node at (7,-0.5) {\tiny $\cdots$};
\node at (8,0.75) {\tiny $n-1$};
%\node at (8.5,0.25) {\tiny $a_1$};
\node at (9,0.75) {\tiny $n$};
\end{tikzpicture}
\caption{Idempotent of the $A$-Brauer algebra when $\delta=0$, and $n$ is odd.}\label{Idempotent of the A-Brauer algebra when delta 0}
\end{figure}

The cellular structure of $A$-Brauer algebra is known from \cite{GG}, we will prove that the algebra $D_n(A)$ can be written as an iterated inflation of the wreath product algebra $A\wr \Sg_n$.  The ``iterated inflation" is an inductive construction of cellular algebra from a smaller cellular algebra, see [\cite{KX3}, Section 3] for more details. Let $0 \leq l \leq \lfloor \frac{n}{2} \rfloor$, where $\lfloor \frac{n}{2}\rfloor$ is  the integer part of $\frac{n}{2}$ and $J_l$ be the free $R$-module generated by all $A$-Brauer diagrams with at least $l$ horizontal edges. Throughout the paper, we assume that the diagram having $l$ horizontal edges means it has $l$ horizontal edges in each row.

\begin{lemma}\label{Jl two sided ideal of A-Brauer algebras}
 If $0 \leq l \leq \lfloor \frac{n}{2} \rfloor$, then $J_l$ is a two-sided ideal of $D_n(A)$. 
\end{lemma}
\begin{proof}
For any $d_1,d_2 \in D_n(A)$ with $m_1 $ and $m_2$ horizontal edges, the product $d_1 \cdot d_2$ have at least $\max\{m_1,m_2\}$ horizontal edges. Thus, $J_l$ is a two-sided ideal, where $0 \leq l \leq \lfloor \frac{n}{2} \rfloor$.
\end{proof}

Let $0 \leq l \leq \lfloor \frac{n}{2} \rfloor$. Consider the chain of two-sided ideals of $D_n(A)$ $$0 \subseteq J_{\lfloor \frac{n}{2} \rfloor} \subseteq \cdots \subseteq J_1 \subseteq J_0=D_n(A),$$ where each $J_l$ is generated by the idempotent $e_l$, which we defined already.
 An \textit{$(n,l)$ $A$-labeled partial diagram} is a one-row diagram on $n$ vertices with $l$ $A$-labeled horizontal edges, and $n-2l$ free vertices. Let $V(n,l,A)$ be the free $R$-module generated by all $(n,l)$ $A$-labeled partial diagrams. From now on, let us assume that $R$ is a commutative Noetherian domain.

%Let $z$ be the new vertical edge formed after multiplying $X$ and $Y$. Then $z$ can be made by concatenating some of the horizontal edges of the top row of $Y$ and some of the horizontal edges of the bottom row of $X$, for example (\ref{}). 
%Define the bilinear form $\phi: V(n,l,A) \times V(n,l,A) \rightarrow R[A\wr\Sg_{n-2l}]$ by \s{check}
%\[
%\phi(e , f)= \begin{cases}\label{bilinear form A-Brauer algebra}
% \prod_i \mathrm{tr}(a_i)^{l_i} \sigma &\text{ if each edges of } f\text{ and }g \\
%& \text{ share atleast one common vertex }\\
%0 & \text{ otherwise}
%\end{cases}
%\]
%where $l_i$ is the number of the closed connected component of labeled $a_i$, while concatenated two partial diagrams $e$ and $f$, the pemutation 

Consider the $R$-algebra $V(n,l,A)\otimes_R V(n,l,A)\otimes_R A\wr \Sg_{n-2l}$ having a basis of the form $e \otimes f \otimes \pi_d$, where $e$ and $f$ are respectively the top and bottom configuration of an $A$-Brauer diagram $d$, and $\pi_d \in A\wr \Sg_{n-2l}$. The multiplication in $V(n,l,A)\otimes V(n,l,A)\otimes A\wr \Sg_{n-2l}$ is defined by $$(e \otimes f \otimes \pi_d)\cdot (e'\otimes f'\otimes \pi_{d^{'}}) = e \otimes f' \otimes \pi_d \phi(f,e') \pi_{d^{'}},$$ where $\phi: V(n,l,A) \otimes_R V(n,l,A) \longrightarrow A\wr \Sg_{n-2l}$ is a bilinear form. Let the $R$-algebra $J_l/J_{l+1}$ be denoted by $B$.

\begin{lemma}\label{each layer corresponds to iterated inflation}
For $0 \leq l \leq \lfloor \frac{n}{2} \rfloor$, $B$ is isomorphic to the inflation $V(n,l, A) \otimes V(n,l,A)\otimes A\wr \Sg_{n-2l}$ of the wreath product algebra along $V(n,l,A)$.
\end{lemma}
\begin{proof}
The $R$-algebra $B$ has a basis consisting of all diagrams with exactly $l$ horizontal edges in the top row, another $l$ horizontal edges in bottom row and $n-2l$ vertical edges. Define an $R$-linear map
\begin{align*}
\Psi: B & \longrightarrow V(n,l,A) \otimes V(n,l,A) \otimes A\wr \Sg_{n-2l}\\
d &\mapsto e \otimes f\otimes \pi_d
\end{align*}
as follows: Fix $d\in B$ and denote by $e$ (resp. $f$) the $(n,l)$ $A$-labeled partial diagram in top (resp. bottom) configuration of $d$. Renumber the free vertices of $e$ and $f$ as $1,2,\cdots, n-2l$ from left to right. Then the vertical edges of $d$ define a permutation $\pi_d\in A \wr\Sg_{n-2l}$ sending the number assigned to top vertex of an edge to the number in its bottom vertex. We need to define a bilinear form $\phi$ to give a multiplicative structure on $V(n,l,A) \otimes V(n,l,A) \otimes A\wr \Sg_{n-2l}$. This construction will use the map $\Psi$, thereby confirming that $\Psi$ is indeed an $R$-algebra isomorphism. Clearly, $\Psi$ is $R$-module isomorphism and we need to check that $\Psi$ preserves the multiplication of two diagrams in $V(n,l,A) \otimes V(n,l,A) \otimes A\wr \Sg_{n-2l}$. Let $d,d' \in B$ with $\Psi(d)=X$, and $\Psi(d')=Y$, where $X$ and $Y$ are of the form $e\otimes f\otimes \pi_d$ and $e'\otimes f' \otimes \pi_{d^{'}}$, respectively. Then the product $X\cdot Y$ has to be of the form $e \otimes f' \otimes \pi_d \phi(f,e')\pi_{d^{'}}$. The product $X\cdot Y$ is obtained by identifying the vertices in $f$ with those in $e$. Now, we have two cases:

\begin{itemize}
\item[(i)] In the concatenated diagram if each horizontal edge of $f$ and $e'$ shares at least one common vertex, then we define the bilinear form $\phi: V(n,l,A) \times V(n,l,A) \rightarrow A\wr\Sg_{n-2l}$ by $\phi(f , e')=\prod_i \mathrm{tr}(a_i)^{l_i}\cdot\sigma$, where $l_i$ is the number of connected components that lies entirely in the middle row of the concatenated diagram and $\sigma$ is the permutation defined by the vertical edges of the product $\Psi^{-1}( f \otimes f \otimes \mathrm{id})\cdot  \Psi^{-1} (e' \otimes e'\otimes \mathrm{id})$. Then the product $X\cdot Y$ coincides with $\Psi(\Psi^{-1}(X) \Psi^{-1}(Y))$ .
\item[(ii)]  If at least one horizontal edges in $f$ which does not share a vertex with any horizontal edges in $e'$, then we define $\phi(f',e)=0$. The number of horizontal edges of $\Psi^{-1}(X)\cdot \Psi^{-1}(Y)$ in $D_n(A)$ is strictly greater than $l$ which implies that the product is zero inside $B$. 
%, which implies $X\cdot Y=0$ inside $V(n,l,A) \otimes V(n,l,A) \otimes A\wr \Sg_{n-2l}$. This is due to the fact that, the number of horizontal edges of $d\cdot d'$ in $D_n(A)$ is strictly greater than $l$. Hence $\Psi(d\cdot d')=0,$ and $\Psi$ is an $R$-algebra isomorphism. 
\end{itemize}
Hence, $\Psi$ is an $R$-algebra isomorphism.
\end{proof}

\begin{remark} 
The $R$–rank of $V(n,l,A)$ is $(\dim A)^n 
\big(\frac{n!}{l!(n-2l)! 2^l}\big)$. 
\end{remark}

\begin{lemma}
Let $0\leq k,l\leq \lfloor \frac{n}{2}\rfloor$. Let $d \in J_k/J_{k+1}$ and $d' \in J_l/J_{l+1}$ be two diagrams in $D_n(A)$, which maps under $\Psi$ to $e \otimes f \otimes \pi_{d}$ and $e'\otimes f' \otimes \pi_{d'}$, respectively. Then the product $d\cdot d'$ is either an element of $J_{s}/J_{s+1}$ or lies in $J_{s+1}$, where $s=\max\{k,l\}$. Moreover, if $\max\{k,l\}=l$, then the product is of the form $c( e''\otimes e'\otimes \sigma \pi_{d'})$, where $c$ is a scalar, $e'' \in V(n,l,A)$, and $\sigma \in A\wr \Sg_{n-2l}$. 
\end{lemma}
\begin{proof}
The proof is similar to that of the Lemma \ref{each layer corresponds to iterated inflation}. 
\end{proof}
In the following lemma, an $R$-linear anti-automorphism $f$ of an $R$-algebra with $f^2=\mathrm id$ will be called an \textit{involution}.
\begin{lemma}
If $0\leq l\leq \lfloor \frac{n}{2}\rfloor$, then the involution $*:A\wr \Sg_{n-2l} \rightarrow A\wr \Sg_{n-2l}$ extends to an involution $i$ on $D_n(A)$ given by $i( e\otimes f\otimes \pi_d)= (f \otimes e \otimes \pi_{d}^{*})$.
\end{lemma}
\begin{proof}
Straightforward. 
\end{proof}
As an immediate consequence, we have the following:
\begin{theorem}\label{iterated inflation of A-brauer algebra}
The algebra $D_n(A)$ is an iterated inflation of  wreath product algebras $A\wr \Sg_{n-2l}$ with the following decomposition 
\begin{equation*}
D_n(A) \cong \bigoplus_{l=0}^{\lfloor \frac{n}{2} \rfloor} V(n,l,A) \otimes V(n,l,A) \otimes A\wr \Sg_{n-2l}.
\end{equation*}
\end{theorem}

As a consequence of this theorem and [\cite{KX3}, Proposition 3.5], we recover the following result from [\cite{GG}, Theorem 6.1].

\begin{corollary}
Let $A$ be a cyclic cellular algebra. Then, for $n \geq 1$, $D_n(A)$ is a cyclic cellular algebra.
\end{corollary}
\begin{proof}
The assumption on $A$ ensures that the wreath product algebra $A\wr\Sg_n$ is cyclic cellular. Therefore, the $A$-Brauer algebra $D_n(A)$ is cyclic cellular by [\cite{KX3}, Proposition 3.5].
\end{proof}

%\begin{proposition}
%The $A$-Brauer algebra $D_n(A)$ is morita equivalent to $\bigoplus_{l=0}^{\lfloor \frac{n}{2}\rfloor} R[A\wr \Sg_{n-2l}$.
%\end{proposition}
%\begin{proof}
%
%\end{proof}

\subsection{Exact split pair ($D_n(A),A\wr \Sg_n$)}
The main objective of this article is to establish the sufficient condition for the existence of an exact split pair for the case of $A$-Brauer algebras using the method from \cite{DK}. Constructing an exact split pair involves identifying a corner split quotient, as outlined in the Theorem \ref{general corner split quotient}.
\begin{proposition} \label{split quotient for A-Brauer algebra}
The algebra $A\wr \Sg_n$ is a split quotient of $D_n(A)$. 
\end{proposition}
\begin{proof}
The subalgebra $A\wr \Sg_n$ of $D_n(A)$ is a free $R$-module generated by $s_{i}$, $h_j^a$ for $1\leq i\leq n-1$, $1 \leq j \leq n$, and $a \in A$. Let $\epsilon$ be a canonical inclusion of $A\wr\Sg_n$ into $D_n(A)$, where the unit element of $A\wr \Sg_n$ is mapped to the unit element of $D_n(A)$. Let us define a map $\pi: D_n(A) \twoheadrightarrow A\wr\Sg_n$ in such a way that $\pi$ takes each $d\in D_n(A)$ to a diagram that has no horizontal edges. Since $\pi$ is surjective, the $\mathrm{kernel}$ of $\pi$ contains all $A$-Brauer diagrams with at least one horizontal edge. Therefore, $\ker \pi$ coincides with $J_1$, and $D_n(A)/J_1 \cong A\wr \Sg_n$. Thus, $\pi \circ \epsilon$ is the identity on $A\wr \Sg_n$. 
\end{proof}
From now on, assume $\delta$ is invertible in $R$, where $R$ is a commutative Noetherian domain. 
%
% These maps induce the functors $F$ and $G$ by $F=- \otimes_{D_n(A)} D_n(A)e_1 $ and $G= - \otimes_{R[A\wr \Sg_n]} e_1(D_n(A)/J_{1})$. It is evident that the composition $F\circ G$ acts as an identity on $R[A\wr \Sg_n]$. \s{check at the last, at this point we can ignore this}

\begin{proposition}\label{corner rings for A Brauer algebra}
If $0 \leq l \leq \lfloor\frac{n}{2}\rfloor$, then $D_{n-2l}(A)$ is isomorphic to the corner ring $e_lD_n(A) e_l$.
\end{proposition}
\begin{proof}
The algebra $e_l D_n(A)e_l$ has a basis consisting of diagrams in $D_n(A)$ with at least $l$ horizontal edges in the top (resp. bottom) row placed at the right side of the diagram. Consider the map $\phi$ defined by
\begin{align*}
\phi: D_{n-2l}(A) \longrightarrow &e_l D_n(A)e_l\\
d \longmapsto & e_lde_l,
\end{align*}
where we identify $d$ inside $D_n(A)$ by adding vertical edges from the vertices $j$ in top row to the vertices $j$ in bottom row, where $n-2l+1 \leq j \leq n$. Clearly, $\phi$ is a surjective homomorphism, and $\ker(\phi)$ contains only the zero element. This implies $\phi$ is injective, and hence, $D_{n-2l}(A) \cong e_l D_n(A)e_l$.
\end{proof}

\begin{theorem}\label{corner split quotient for A Brauer algebra}
If $0 \leq l \leq \lfloor \frac{n}{2} \rfloor$, then the algebra $A\wr \Sg_{n-2l}$ is a corner split quotient of $D_n(A)$.
\end{theorem}
\begin{proof}
 By Proposition \ref{split quotient for A-Brauer algebra}, $A\wr \Sg_{n-2l}$ is a split quotient of $D_{n-2l}(A)$ and is isomorphic to the corner ring $e_l D_n(A) e_l$ by Proposition \ref{corner rings for A Brauer algebra}. Therefore, we only need to show that there is an $e_lD_n(A)e_l$-$D_n(A)$-bimodule $S_l$, which is projective as a left $A\wr \Sg_{n-2l}$-module and satisfies $S_le_l \cong A\wr \Sg_{n-2l}$ as a right $A\wr \Sg_{n-2l}$-modules.

As the right $D_n(A)$-module, $e_lD_n(A)$ is a projective module generated by $A$-Brauer diagrams with at least $l$ horizontal edges, where the top row contains the partial diagram of $e_l$. Recall that $e_lD_n(A)e_l$ consists of $A$-Brauer diagrams having at least $l$ horizontal edges with the partial diagram of $e_l$ on its top and bottom configurations. Therefore, $e_lD_n(A)$ is a left $e_lD_n(A)e_l$-module. Now, the right action of $e_lD_n(A)e_l$ on $A\wr \Sg_{n-2l}$ is given as follows: let $d \in D_n(A)$, and $\sigma \in A\wr \Sg_{n-2l}$
\begin{equation*}
\sigma \cdot d =
 \begin{cases}
  \sigma d &; \text{ if } d \text{ has exactly  } l \text{ horizontal edges}\\
 0 &;\text{ otherwise. } %d \in \ker \alpha
 \end{cases}
\end{equation*}
If $d$ has more than $l$ horizontal edges, then the quotient map $ \alpha: e_l D_n(A) e_l \longrightarrow e_l (D_n(A)/J_{l+1})e_l ~\big(\cong \frac{e_l D_n(A)e_l}{e_l J_{l+1} e_l} \big)$ is defined by $\alpha(e_l de_l)= e_l d e_l \mod (J_{l+1})$. Clearly, $\ker \alpha$ contains all diagrams having at least $l+1$ horizontal edges. The map $\alpha$ induces an action of $e_lD_n(A)e_l$ on $A\wr \Sg_{n-2l}$. Define the $e_lD_n(A) e_l$-$D_n(A)$-bimodule $S_l$ to be $ A\wr\Sg_{n-2l}\otimes_{e_l D_n(A) e_l}e_l D_n(A)$.

Next, we claim that $S_l$ is a free left $A\wr \Sg_{n-2l}$-module. As an $R$-module, $S_l$ is generated by elements of the form $\sigma \otimes e_ld $, where $d \in D_n(A)$, and $\sigma \in A\wr \Sg_{n-2l}$. If $d$ has more than $l$ horizontal edges, say $m$, then we have 
\begin{align*}
\sigma \otimes e_ld  &=  \sigma \otimes e_md'  \\
&= \sigma e_m  \otimes  d'\\
& = 0 \otimes d' =0,
\end{align*}
where $d' \in D_n(A)$ contains the bottom configuration of $d$. The last equality holds because the idempotent $e_m$ is in $\ker \alpha$. In particular, $A\wr \Sg_{n-2l} \otimes_{e_lD_n(A)e_l} J_{l+1}e_l$ vanishes. Consequently, all the generators of $S_l$ are of the form $\sigma \otimes d$, where $d$ has precisely $l$ horizontal edges containing the partial diagram of $e_l$ at its top, and $l$ horizontal edges at the bottom can be arranged arbitrarily. That is, the partial diagram of its bottom configuration lies in $V(n,l,A)$. Therefore, the generators must be of the form $ 1_{A\wr \Sg_{n-2l}} \otimes \sigma d.$ Hence, $S_{l} = (A\wr \Sg_{n-2l})^{\oplus{n_l}}$ where $n_l$ is the rank of $V(n,l,A)$. Thus, $S_l$ is a free $A \wr\Sg_{n-2l}$-$D_n(A)$-bimodule. Thus, $S_{l}$ is projective as a left $A\wr \Sg_{n-2l}$-module.

On the other hand, the right action of $e_l D_n(A)e_l$ on $A\wr \Sg_{n-2l}$ coincides with the action of $e_l(D_n(A)/J_{l+1})e_l$ on $A\wr \Sg_{n-2l}$. Hence, we have
\begin{align}
S_l&=A\wr\Sg_{n-2l}\otimes_{e_l D_n(A) e_l}e_l D_n(A) \nonumber \noindent\\
&\cong A\wr\Sg_{n-2l}\otimes_{e_l D_n(A) e_l}e_l (D_n(A)/J_{l+1})\nonumber \noindent\\
&\cong A\wr\Sg_{n-2l}\otimes_{e_l (D_n(A)/J_{l+1}) e_l}e_l (D_n(A)/J_{l+1})\nonumber \noindent\label{dirc pdt has exactly l edges for cyclotomic Brauer}\\
&\cong A\wr\Sg_{n-2l}\otimes_{A\wr\Sg_{n-2l}}e_l (D_n(A)/J_{l+1})\\
&\cong e_l (D_n(A)/J_{l+1}),\nonumber \noindent
\end{align}
where (\ref{dirc pdt has exactly l edges for cyclotomic Brauer}) follows the fact that $e_l(D_n(A)/J_{l+1})e_l$ has exactly $l$ horizontal edges so that it will be isomorphic to $A\wr\Sg_{n-2l}$.
% Thus, $S_{l}$ is projective as a left $R[A\wr \Sg_{n-2l}]$-module.
One can easily obtain the following right $A\wr \Sg_{n-2l}$-module isomorphism:
\begin{align*}
S_le_l&\cong\big(e_l (D_n(A)/J_{l+1})\big)e_l\\
&\cong A\wr \Sg_{n-2l}.
\end{align*}
This ends the proof.
\end{proof}

\begin{remark} 
Consider the categories $\text{\textbf{mod}-}A\wr \Sg_{n-2l}$, and $\text{\textbf{mod}-}D_n(A)$. Using the bimodule structure of $S_l$, we define the layer induction functor $\ind_l(M)=M \otimes_{e_lD_n(A)e_l} S_l$, where $M \in \text{\textbf{mod}-}A\wr \Sg_{n-2l}$ and $0 \leq l \leq \lfloor \frac{n}{2} \rfloor$. 
%\begin{align*}
%\ind_l : \text{ \textbf{mod}-}& R[A\wr \Sg_{n-2l}  \longrightarrow \B_n^r(\delta) \text{-\textbf{mod}} \\
%& M \longmapsto  (\B_n^r(\delta)/J_{l+1})e_{l} \otimes_{\h_{n-2l}^r} M .
%\end{align*}
One can define the restriction functor by $\Res_l (N)= N \otimes_{D_n(A)} D_n(A)e_l$ for all $N \in \text{\textbf{mod}-}D_n(A)$ and $\Res_l$ keeps all the layers. 
\end{remark}

\begin{corollary}\label{split pair A Brauer}
For $ 0 \leq l \leq [\frac{n}{2}]$, the pair $(\ind_l,\Res_l)$ forms an exact split pair between $\text{\textbf{mod}-} A\wr \Sg_{n-2l} $ and $\text{\textbf{mod}-}D_n(A)$.
\end{corollary}
\begin{proof}
If $0 \leq l \leq \lfloor \frac{n}{2} \rfloor$, then $A\wr \Sg_{n-2l}$ is a corner split quotient of $D_n(A)$ by Theorem \ref{corner split quotient for A Brauer algebra}. By applying Theorem \ref{general corner split quotient}, $(\ind_l, \Res_l)$ form an exact split pair of functors between $\text{\textbf{mod}-}D_n(A)$ and $\text{\textbf{mod}-}A\wr \Sg_{n-2l}$.
\end{proof}
Note that in the case of $\delta=0$ and $n$ is odd, the statements in Proposition \ref{split quotient for A-Brauer algebra}, Proposition \ref{corner rings for A Brauer algebra} and Corollary \ref{split pair A Brauer} are still valid.

\section{The Cyclotomic Brauer algebra}\label{Cyclotomic Brauer algebra}
 For fixed natural numbers $r$ and $n$, let $R$ be a commutative ring containing a primitive $r$-th root of unity and parameters $\delta_0,\delta_1, \cdots , \delta_{r-1}$.  Let $\bdel=(\delta_0 , \cdots, \delta_{r-1})\in R^{r}$ and $A$ be the group algebra of the cyclic group $\Z/r\Z$. The associated $A$-Brauer algebra is known as the cyclotomic Brauer algebra $\B_n^r(\bdel)$. These algebras were studied in \cite{BCV,RX,RY1}. Similar to the $A$-Brauer diagrams, all the edges are equipped with an orientation. Reversing the orientation results in relabeling the edges using the inverse element in $\Z/r\Z$. It is generated by $s_{i}$, $e_{i}$ and $h_{j}^{m}$ subject to the relations [\cite{RX}, Section 2], for $1\leq i \leq n-1$, $1 \leq j \leq n$ and $m \in \Z/r\Z$. 

Let $0 \leq l  \leq \lfloor\frac{n}{2} \rfloor$. If $\bdel$ is invertible, then we consider the idempotent as in Figure \ref{Idempotent of the A-Brauer algebra when delta is nonzero}. If $\bdel=0$ and $n$ is odd, then the idempotent is given as in Figure \ref{Idempotent of the A-Brauer algebra when delta 0}. 
%One can define the idempotents to be the idempotent of $A$-Brauer algebras. For $a_i \in \Z/r\Z$, define $\mathrm{tr}(a_i)= \bdel_i$ for $i=\{0,\cdots, r-1\}$. When $\bdel \neq 0$ then we can define an idempotent by Figure \ref{Idempotent of the A-Brauer algebra}. 

%\s{But, when $\bdel =0$, in this case the idempotent is defined by Figure \ref{Idempotent of A-Brauer algebras}. }

\subsection{The Generalized symmetric group} \label{Generalized symmetric group}

Let $\h_{n}^{r}$ be the group algebra of the generalized symmetric group $\Z/r\Z \wr\Sg_n$. It is a subalgebra of $\B_n^r(\bdel)$ generated by $s_{i}$, and $h_{j}^{m}$ for $1 \leq i \leq n-1$, $1 \leq j \leq n$, and $m \in \Z/r\Z$. The group algebra  $\h_{n}^{r}$ is cyclic cellular with cell datum $(\h_{n}^{r}, *, \Lambda_n^{r}, \unrhd, \mathcal{M}, \mathcal{B} )$, see [\cite{GG}, Theorem 4.1] where $\Lambda_n^{r}$ is the set of all multipartitions of $n$ with $r$ parts which is a poset under the dominance order for the multipartitions.  

%The generalized symmetric group $\h_{n}^{r}$ is the group algebra $A \wr \Sg_n$, where $A$ is the cyclic group $\Z/r\Z$. It is a subalgebra of $D_n(A)$ generated by $s_{i}$, and $h_{j}^{m}$ for $1 \leq i \leq n-1$, $1 \leq j \leq n$, and $m \in \Z/r\Z$ . The group algebra  $\h_{n}^{r}$ is cyclic cellular with cell datum $(\h_{n}^{r}, *, \Lambda_n^{r}, \unrhd_{\Gamma}, \mathcal{M}, \mathcal{B} )$, see [\cite{Graham}, section 5] and [\cite{GG}, Theorem 4.1] where $\Lambda_n^{r}$ is the set of all multipartitions of $n$ with $r$ parts which is a poset under the dominance order for the multipartitions, see [\cite{Graham}, section 5]. 
 %For $\bolam \in\Lambda_n^{r}$, the cell modules of $\h_n^r$ are identified with the dual Specht modules $S_{\bolam}$ see [\cite{DR2}, Theorem 2.9].

From now on, assume $R$ to be a commutative Noetherian domain with an invertible $\bdel$. Let $V_l$ be the free $R$-module generated by all $(n,l)$-$\Z/r\Z$ labeled partial diagrams. 
The algebra $\B_n^r(\bdel)$ can be decomposed as an iterated inflation of $\h_{n-2l}^r$ along $V_l$, 
\begin{equation*}
 \B_n^r(\bdel) \cong \bigoplus_{l=0}^{\lfloor\frac{n}{2}\rfloor} V_l\otimes V_l \otimes \h_{n-2l}^r,
\end{equation*}
where $\h_{n-2l}^r$ is a cellular algebra, see [\cite{BCV}, Theorem 3.1.3]. The involution is given by reflecting the diagram along its horizontal axis.

Define the index set $\Lambda$ by $\Lambda:= \{ (l,\bolam): 0 \leq l \leq \lfloor\frac{n}{2}\rfloor, \bolam \in \Lambda_n^r\}$. We extend the dominance order of $\h_n^r$ to $\B_n^r(\bdel)$ to get the dominance order ``$\geq$" on $\Lambda$, which is defined as follows: 
\begin{equation*}
(l,\bolam) \geq (m,\bomu) \text{ if either } l< m, \text{ or } l=m \text{ and } \bomu \trianglerighteq \bolam.
\end{equation*}

The cell modules for $\B_n^r(\bdel)$ are indexed by $(l, \bolam) \in \Lambda$ by [\cite{RY1}, Theorem 5.11].

\subsection{Exact split pair ($\B_n^r(\bdel), \h_n^r$)}\label{exact split pair for cyclotomic Brauer algebra}
In this subsection, we show that there exists an exact split pair of functors between $\text{\textbf{mod}-}\B_n^r(\bdel)$ and $\text{\textbf{mod}-}\h_n^r$. As an application of this, we obtain the global dimension of $\B_n^r(\bdel)$. 

Let us define the induction and restriction functor as follows: 
\begin{align*}
 \ind_l:\text{\textbf{mod}-}& \h_{n-2l}^r  \longrightarrow  \text{\textbf{mod}-} \B_n^r(\bdel)   &\Res_l :   \text{\textbf{mod}-} \B_n^r(\bdel) \longrightarrow \text{\textbf{mod}-}\h_{n-2l}^r \\
& M \longmapsto  M \otimes_{e_l \B_n^r(\bdel) e_l} e_l\B_n^r(\bdel) &N \longmapsto  N\otimes_{\B_n^r(\bdel)} \B_n^r(\bdel)e_l.
\end{align*}
%where $J_l$ is the two-sided ideal of $\B_n^r(\bdel)$ consisting of diagrams with at least $l$ horizontal edges.
\begin{corollary}\label{exact split pair cyclotomic brauer algebra for split pair}
  The pair $(\ind_l, \Res_l)$ forms an exact split pair of functors between $\text{\textbf{mod}-}\B_n^r(\bdel)$ and $\text{\textbf{mod}-}\h_{n-2l}^r$, where $0 \leq l \leq \lfloor\frac{n}{2} \rfloor$.
\end{corollary}
\begin{proof}
By Theorem \ref{corner split quotient for A Brauer algebra}, $\h_{n-2l}^r$ is a corner split quotient of $\B_n^r(\bdel)$. Therefore, by Theorem \ref{general corner split quotient}, the pair of functors $(\ind_l, \Res_l)$ form an exact split pair.
\end{proof}

% \begin{corollary}\label{corner split quotient for cyclotomic brauer}
%     Let $0 \leq l \leq \lfloor \frac{n}{2} \rfloor$ and  $\bdel \in K^n$. If $\bdel \neq 0$, then the group algebra of the generalized symmetric group $\h_{n-2l}^r$ is a corner split quotient of $\B_n^r(\bdel)$.
% \end{corollary}
% \begin{proof}
% The proof can be followed exactly the same way as Theorem \ref{corner split quotient for A Bruaer algebra}
% \end{proof}
The global dimension of an algebra $A$, denoted by $\mathrm{gl.dim ~}A$, is the maximum of the projective dimensions of the $A$-modules. The computation of the global dimension of $\B_n^r(\bdel)$ is challenging due to its non-self-symmetric nature, the existence of an exact split pair simplifies the calculation.
\begin{proposition}\label{global dimension of cyclotomic Brauer algebra}
If $R$ is a field of characteristic $p$, then $\B_n^r(\bdel)$ has finite global dimension if and only if $p>n$ and $p \nmid r$, or $p=0$.
\end{proposition}
\begin{proof}
Assume that $\B_n^r(\bdel)$ has finite global dimension. It follows that $\h_n^r$ also has finite global dimension because $\mathrm{gl.dim ~} \B_n^{r}(\bdel) \geq \mathrm{gl.dim ~}\h_n^r$ by [\cite{DK}, Corollary 1.6] and Corollary \ref{exact split pair cyclotomic brauer algebra for split pair}. This implies that $p>n$ and $p \nmid r$, or $p=0$. Conversely, if $p>n$ and $p \nmid r$ or $p=0$, then $\B_n^r(\bdel)$ is quasi-hereditary by [\cite{BCV}, Theorem 3.3.1], and hence has finite global dimension by [\cite{KX2}, Theorem 1.1]. 
\end{proof}
Note that in the case of $\bdel=0$ and $n$ is odd, the statements in Corollary \ref{exact split pair cyclotomic brauer algebra for split pair} and Proposition \ref{global dimension of cyclotomic Brauer algebra} are still valid.
\begin{remark} 
The Brauer algebra of type $C$, denoted by $\B(C_n,\delta)$, is an associative $R$-algebra over a field $R$ with $\delta$ a non-zero element in $R$. It was first introduced in \cite{CLY}, and can be written as an iterated inflation of group algebras of hyper-octahedral groups, for more details see \cite{Bo}. The hyper-octahedral group is a special case of the generalized symmetric group when we take $r=2$. The diagram in $\B(C_n,\delta)$ can be viewed as ``unfolding" the diagram vertically along the wall. Similarly one can prove that $R[\Sg_2 \wr \Sg_{n-l}]$ is a corner split quotient of $\B(C_n,\delta)$, which gives an exact split pair $(\ind_l,\Res_l)$ of functors between the categories $\textbf{\text{mod}-}R[\Sg_2 \wr \Sg_{n-l}]$ and $\textbf{\text{mod}-}\B(C_n,\delta)$, where $0 \leq l \leq n$. Note that $\B(C_n,\delta)$ also has finite global dimension if and only if $\delta \neq 0 $ and $p>n$.
\end{remark}

\section{The Walled Brauer algebra}\label{Walled Brauer algebras}

%\subsection{Essential Components}

Let $R$ be a commutative ring with $1$, and $\delta$ a distinguished element in $R$. Fix two positive integers $r$ and $t$. The walled Brauer algebra $\B_{r,t}(\delta)$ is a subalgebra of the ordinary Brauer algebra $\B_{r+t}(\delta)$, and if either $r$ or $t$ is $0$, then $\B_{r,t}(\delta)$ coincides with the ordinary Brauer algebra. An \textit{$(r,t)$-walled Brauer diagram} is an ordinary Brauer diagram with $r+t$ vertices arranged in two rows, separated by a wall between the $r$th and $(r+1)$th vertices. All horizontal edges cross the wall, but vertical edges do not. The walled Brauer algebra $\B_{r,t}(\delta)$ is an associative $R$-algebra spanned by all $(r,t)$-walled Brauer diagrams, with multiplication similar to that of ordinary Brauer algebras. It contains the group algebra of the direct product of two symmetric groups $\Sg_r$ and $\Sg_t$ as a subalgebra and is denoted by $R\Sg_{r,t}$. The generators are $s_i$, $s_{r+j}$, and $e_{k,l}$, where $1 \leq i \leq r-1$, $1 \leq j \leq t-1$, and $1 \leq k \leq r$, $r+1 \leq l \leq r+t$, see Figure \ref{generators of walled Brauer algebra}.

\begin{figure}[H]
\begin{center}
\begin{tikzpicture}[x=0.75cm,y=0.75cm]
\draw[-] (0,0)-- (0,-1);
\draw[-] (1,0)-- (1,-1);
\draw[-] (2,0)-- (3,-1);
\draw[-] (3,0)-- (2,-1);
\draw[-] (4,0)-- (4,-1);
\draw[-] (4.5,0.5)-- (4.5,-1.5);
\draw[-] (5,0)-- (5,-1);
\draw[-] (6,0)-- (6,-1);
%\draw[\mid] (3.5,0.5)-- (3.5,-1.5);
\node at (-1,-0.5) {$s_i=$};
\node at (0,0.5) {\tiny $1$};
\node at (1,0.5) { \tiny $2$};
\node at (1.5,-0.5) {\tiny  $\cdots$};
\node at (2,0.5) {\tiny  $i$};
\node at (3,0.5) { \tiny $i+1$};
\node at (4,0.5) {\tiny  $r$};
\node at (5,0.5) { \tiny $r+1$};
\node at (5.5,-0.5) {\tiny  $\cdots$};
\node at (6,0.5) { \tiny $r+t$};
\end{tikzpicture}
\begin{tikzpicture}[x=0.75cm,y=0.75cm]
\draw[-] (0,0)-- (0,-1);
\draw[-] (1,0)-- (1,-1);
\draw[-] (1.5,0.5)-- (1.5,-1.5);
\draw[-] (2,0)-- (2,-1);
\draw[-] (3,0)-- (4.5,-1);
\draw[-] (4.5,0)-- (3,-1);
\draw[-] (6.5,0)-- (6.5,-1);
\node at (-1,-0.5) {$s_{r+j}=$};
\node at (0,0.5) {\tiny $1$};
\node at (0.5,-0.5) {\tiny $\cdots$};
\node at (1,0.5) {\tiny $r$};
\node at (2,0.5) {\tiny $r+1$};
\node at (2.5,-0.5) {\tiny $\cdots$};
\node at (3,0.5) {\tiny $r+j$};
\node at (4.5,0.5) {\tiny $r+j+1$};
\node at (6,-0.5) {\tiny $\cdots$};
\node at (6.5,0.5) {\tiny $r+t$};
\end{tikzpicture}
\begin{tikzpicture}[x=0.75cm,y=0.75cm]
\draw[-] (0,0)-- (0,-1);
\draw[-] (1,0)-- (1,-1);
\draw[-] (2,0).. controls(4,-0.5)..(6,0);
\draw[-] (3,0)-- (3,-1);
\draw[-] (5,0)-- (5,-1);
\draw[-] (4,0.5)-- (4,-1.5);
\draw[-] (7,0)-- (7,-1);
\draw[-] (2,-1).. controls(4,-0.5)..(6,-1);
\draw[-] (7,0)-- (7,-1);
%\draw[\mid] (3.5,0.5)-- (3.5,-1.5);
\node at (-1,-0.5) {$e_{k,l}=$};
\node at (0,0.5) {\tiny $1$};
\node at (1,0.5) {\tiny $2$};
\node at (2,0.5) {\tiny $k$};
\node at (1.5,-0.5) {\tiny $\cdots$};
\node at (3,0.5) {\tiny $r$};
\node at (5,0.5) {\tiny $r+1$};
\node at (6.5,-0.5) {\tiny $\cdots$};
\node at (6,0.5) {\tiny $l$};
\node at (7,0.5) {\tiny $r+t$};
\end{tikzpicture}
\caption{Generators of the walled Brauer algebra.}\label{generators of walled Brauer algebra}
\end{center}
\end{figure}
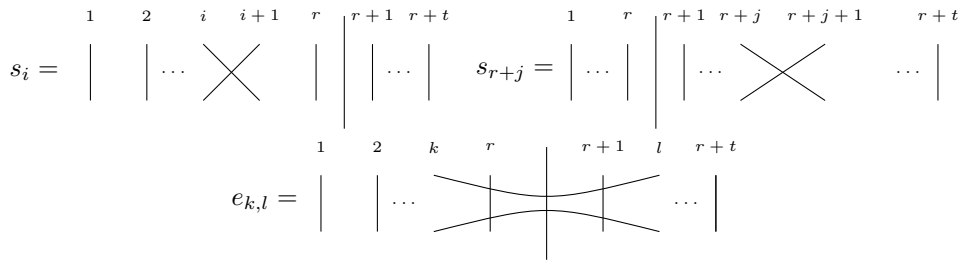

Let $s = \min\{r, t\}$ and $0 \leq l \leq s$. When $\delta $ is invertible in $R$, the idempotent is defined as shown in Figure \ref{Idempotent when delta non-zero}. However, when $\delta = 0$ and either $r$ or $t$ is at least 2, the idempotent is defined as in Figure \ref{Idempotent when delta 0}.

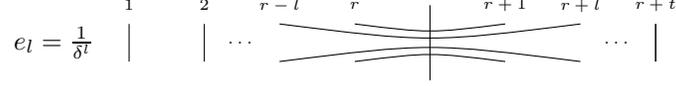
\begin{figure}[H]
\begin{tikzpicture}[x=1cm,y=0.5cm]
\draw[-] (0,0)-- (0,-1);
\draw[-] (1,0)-- (1,-1);
\draw[-] (2,0).. controls(4,-0.5)..(6,0);
\draw[-] (3,0).. controls(4,-0.25)..(5,0);
\draw[-] (4,0.5)-- (4,-1.5);
\draw[-] (7,0)-- (7,-1);
\draw[-] (2,-1).. controls(4,-0.5)..(6,-1);
\draw[-] (3,-1).. controls(4,-0.75)..(5,-1);
\draw[-] (7,0)-- (7,-1);
%\draw[\mid] (3.5,0.5)-- (3.5,-1.5);
\node at (-1,-0.5) {$e_l=\frac{1}{\delta^{l}}$};
\node at (0,0.5) {\tiny $1$};
\node at (1,0.5) {\tiny $2$};
\node at (2,0.5) {\tiny $r-l$};
\node at (1.5,-0.5) {\tiny $\cdots$};
\node at (3,0.5) {\tiny $r$};
\node at (5,0.5) {\tiny $r+1$};
\node at (6.5,-0.5) {\tiny $\cdots$};
\node at (6,0.5) {\tiny $r+l$};
\node at (7,0.5) {\tiny $r+t$};
\end{tikzpicture}
\caption{Idempotent of $\B_{r,t}(\delta)$ for invertible $\delta$.}\label{Idempotent when delta non-zero}
\end{figure}
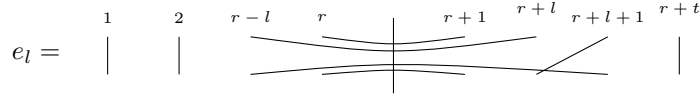
\begin{figure}[H]
\begin{tikzpicture}[x=0.95cm,y=0.5cm]
\draw[-] (0,0)-- (0,-1);
\draw[-] (1,0)-- (1,-1);
\draw[-] (2,0).. controls(4,-0.5)..(6,0);
\draw[-] (2,-1).. controls(4,-0.65)..(7,-1);
\draw[-] (3,0).. controls(4,-0.25)..(5,0);
\draw[-] (4,0.5)-- (4,-1.5);
\draw[-] (7,0)-- (6,-1);
\draw[-] (8,0)-- (8,-1);
\draw[-] (3,-1).. controls(4,-0.85)..(5,-1);
%\draw[-] (7,0)-- (7,-1);
%\draw[\mid] (3.5,0.5)-- (3.5,-1.5);
\node at (-1,-0.5) {$e_l=$};
\node at (0,0.5) {\tiny $1$};
\node at (1,0.5) {\tiny $2$};
\node at (2,0.5) {\tiny $r-l$};
%\node at (1.5,-0.5) {\tiny $\cdots$};
\node at (3,0.5) {\tiny $r$};
\node at (5,0.5) {\tiny $r+1$};
%\node at (6.5,-0.5) {\tiny $\cdots$};
\node at (6,0.75) {\tiny $r+l$};
\node at (7,0.5) {\tiny $r+l+1$};
\node at (8,0.75) {\tiny $r+t$};
\end{tikzpicture}
\caption{Idempotent of $\B_{r,t}(\delta)$ when $\delta= 0$, and one of $r$ or $t$ $\geq 2$.}\label{Idempotent when delta 0}
\end{figure}

\subsection{Cellular structure}
Let $s:=\min\{r,t\}$, and $0 \leq l \leq s$. Assume that $R$ is a commutative Noetherian domain with invertible $\delta$. Let $V_l$ be the free $R$-module generated by all $(r,t)$-partial diagrams with $l$ horizontal edges, $r-l$ free vertices on the left side, and $t-l$ free vertices on the right side of the wall. The involution $i$ is defined by flipping the diagram along the horizontal axis. By [\cite{CVDM}, Proposition 2.6], the algebra $\B_{r,t}(\delta)$ can be expressed as an iterated inflation of $R\mathfrak{S}_{r-l,t-l}$ along $V_l$: $$ \mathcal{B}_{r,t}(\delta) =\bigoplus_{l=0}^{s} V_{l}\otimes V_{l} \otimes R\mathfrak{S}_{r-l,t-l}.$$ The group algebra $R\Sg_{r-l,t-l}$ being a direct product of two cellular algebras is cellular by [\cite{GM}, Example 3.1.6]. Therefore, by [\cite{KX3}, Proposition 3.5], $\B_{r,t}(\delta)$ is also cellular.

% The following proposition ensures that $\B_{r,t}(\delta)$ is cellular by [\cite{KX3}, Proposition 3.5].
% \begin{proposition} [\cite{CVDM}, Proposition 2.6]
% The algebra $\B_{r,t}(\delta)$ can be written as an iterated inflation of group algebras of the direct product of the symmetric groups $\mathfrak{S}_{r-l,t-l}$ along the vector space $V_l$ over $K$ $$ \mathcal{B}_{r,t}(\delta) =\bigoplus_{l=0}^{s} V_{l}\otimes V_{l} \otimes K\mathfrak{S}_{r-l,t-l}.$$
% \end{proposition} 

% Define a poset $\Lambda$ as follows: $\Lambda:= \{ (l,(\lambda,\mu)): 0 \leq l \leq s, \text{ and } (\lambda, \mu) \in \Lambda_{r-l,t-l} \}.$ For any $(\lambda, \mu), (\lambda', \mu') \in \Lambda_{r-l,t-l}$, we can define a partial order on $\Lambda$ as $$(l,(\lambda,\mu)) \leq (m,(\lambda',\mu')) \text{ if } m \leq l \text{ or } m=l \text{ and } (\lambda, \mu) \unrhd (\lambda',\mu').$$ The cell modules of $\B_{r,t}(\delta)$ are indexed by the elements $(l,(\lambda,\mu))$ of $\Lambda$ by [\cite{CVDM}, Proposition 2.7(i)].

\subsection{ Exact split pair ($\B_{r,t}(\delta), R\Sg_{r,t}$)}\label{Exact split pair situation in the case of walled Brauer algebras}
We now show the existence of an exact split pair for $\B_{r,t}(\delta)$. This will allow us to determine the criteria for the non-vanishing of $\Hom$ and $\Ext$ spaces between the cell modules of $\B_{r,t}(\delta)$.

Let $0 \leq l \leq s$. Define the induction functor $\ind_l$, and restriction functor $\Res_l$ as follows: $\ind_l(M)=  M \otimes_{e_l \B_{r,t}(\delta)e_l} e_l\B_{r,t}(\delta)$ for $M \in \text{\textbf{mod}-} R\Sg_{r-l,t-l}$ and $\Res_l(N)=N\otimes_{\B_{r,t}(\delta)} \B_{r,t}(\delta)e_l $ for $N \in \text{\textbf{mod}-} \B_{r,t}(\delta)$.
%, where $J_l$ is the two-sided ideal of $\B_{r,t}(\delta)$ consisting of $(r,t)$-walled Brauer diagrams having at least $l$ horizontal edges. 

\begin{theorem}\label{corner split quotient of walled Brauer algebras}
If $\delta$ is invertible, then for each $l \in \{1, \cdots, s\}$, there exists an exact split pair of functors $(\ind_l, \Res_l)$ between the categories $\textbf{\text{mod}-}R\Sg_{r-l,t-l}$, and $\textbf{\text{mod}-}\B_{r,t}(\delta)$.
\end{theorem}
\begin{proof}
The proof uses a similar argument as in Corollary \ref{exact split pair cyclotomic brauer algebra for split pair}. 
\end{proof}
Next, we will show some applications of these exact split pairs. The poset $\Lambda$ is the set of all tuples $(l, (\lambda, \mu))$ where $0 \leq l \leq s$ and $\lambda, \mu$ are partitions of $r-l$ and $t-l$, respectively. Let $(l,(\lambda, \mu)), (m,(\lambda', \mu')) \in \Lambda$, the partial order ``$\leq $" on $\Lambda$ is defined as follows: $$(l,(\lambda,\mu)) \leq (m,(\lambda',\mu')) \text{ if } m < l, \text{ or } m=l \text{ and } \lambda \unrhd \lambda',  \mu \unrhd \mu'.$$ Therefore, by using the induction functor, we write the cell modules of $\B_{r,t}(\delta)$ as $\ind_l S^{\lambda,\mu}$ where $(l,(\lambda,\mu)) \in \Lambda$ and $S^{\lambda,\mu}$ is the Specht module of $R\Sg_{r-l,t-l}$. From now on, we assume $R$ to be a field. 

% \begin{remark} 
% Let $\text{\textbf{mod}-}K\Sg_{r-l,t-l}$ denote the category of all right $K\Sg_{r-l,t-l}$-modules, and $\text{\textbf{mod}-}\B_{r,t}(\delta)$ represent the category of all right $\B_{r,t}(\delta)$-modules. For $0 \leq l \leq s$, define the induction and restriction functors by 
% %$\ind_l(M)= M \otimes_{e_l\B_n^r(\delta)e_l} e_l \B_n^r$ for $M \in \text{\textbf{mod}-}\h_{n-2l}^r$. 
% \begin{align*}
%  \ind_l:\text{\textbf{mod}-}& K\Sg_{r-l,t-l}  \longrightarrow  \text{\textbf{mod}-} \B_{r,t}(\delta)   &\Res_l :   \text{\textbf{mod}-} \B_{r,t}(\delta) \longrightarrow \text{\textbf{mod}-}K\Sg_{r-l,t-l} \\
% & M \longmapsto  M \otimes_{\substack{K\Sg_{r-l,t-l}}} e_l(\B_{r,t}(\delta)/J_{l+1})  &N \longmapsto  N\otimes_{\B_{r,t}(\delta)} \B_{r,t}(\delta)e_l 
% \end{align*}
% The pair of functors $(\ind_l,\Res_l)$ forms an exact split pair.
% \end{remark}

\begin{corollary} \label{Hom between cell module for walled Brauer for split pair}
If $ \mathrm{char} ~R \neq 2$, then $\Hom_{\B_{r,t}(\delta)}( \ind_l S^{\lambda,\mu}, \ind_{l} S^{\lambda',\mu'}) \neq 0$ implies that $\lambda \triangleright \lambda'$, and $\mu \triangleright \mu'$. 
\end{corollary}
\begin{proof}
We have the following isomorphism using [\cite{DK}, Proposition 1.3], and Theorem \ref{corner split quotient of walled Brauer algebras}
\begin{align*}
\Hom_{\B_{r,t}(\delta)}( \ind_l S^{\lambda,\mu}, \ind_{l} S^{\lambda',\mu'}) &\cong \Hom_{R\Sg_{r-l,t-l}}( S^{\lambda,\mu},  S^{\lambda',\mu'})\\
& \cong \Hom_{R\Sg_{r-l}} (S^{\lambda}, S^{\lambda'}) \boxtimes_R \Hom_{R\Sg_{t-l}} (S^{\mu}, S^{\mu'}),
\end{align*}
and the last isomorphism follows from [\cite{Xi}, Lemma 3.2]. If $\lambda \ntrianglerighteq \lambda'$ or $ \mu \ntrianglerighteq \mu'$, then the last isomorphism is zero by [\cite{JaB}, Corollary 13.17], provided $\mathrm{char}~R\neq 2$.
% This implies that one of the tensors must be zero.  If $\mathrm{char}~R\neq 2$, then [\cite{JaB}, Corollary 13.17] implies that each tensor is also zero whenever $\lambda \ntrianglerighteq \lambda'$ or $ \mu \ntrianglerighteq \mu'$.
 Hence, $\Hom_{\B_{r,t}(\delta)}( \ind_l S^{\lambda,\mu}, \ind_{l} S^{\lambda',\mu'}) \neq 0$ implies $\lambda \triangleright \lambda'$ and $\mu \triangleright \mu'$. 
\end{proof}

\begin{corollary}\label{Ext between cell module for walled for split pair}
    If $ \mathrm{char} ~R \neq 2,3$, then $\Ext_{\B_{r,t}(\delta)}^1( \ind_l S^{\lambda,\mu}, \ind_{l} S^{\lambda',\mu'}) \neq 0$ implies $\lambda \trianglerighteq \lambda'$ and $\mu \unrhd \mu'$. 
\end{corollary}
\begin{proof}
We have the following isomorphism from [\cite{DK}, Proposition 1.6], and Theorem \ref{corner split quotient of walled Brauer algebras}
\begin{align*}
&\Ext_{\B_{r,t}(\delta)}^1( \ind_l S^{\lambda,\mu}, \ind_{l} S^{\lambda',\mu'}) \cong \Ext_{R\Sg_{r-l,t-l}}^1( S^{\lambda,\mu},  S^{\lambda',\mu'})\\
& \cong \big(\Ext_{R\Sg_{r-l}}^1 (S^{\lambda}, S^{\lambda'}) \boxtimes_R \Hom_{R\Sg_{t-l}} (S^{\mu}, S^{\mu'})\big) \bigoplus \big(\Hom_{R\Sg_{r-l}} (S^{\lambda}, S^{\lambda'}) \boxtimes_R \Ext_{R\Sg_{t-l}}^1 (S^{\mu}, S^{\mu'})\big),
\end{align*}
and the last isomorphism follows from [\cite{CE}, Chapter XI, Theorem 3.1]. This first direct summand is zero whenever $\lambda \ntriangleright \lambda'$ or $\mu \ntrianglerighteq \mu'$ by [\cite{HN}, Theorem 4.2.1] and 
[\cite{JaB}, Corollary 13.17], provided $\mathrm{char}~R\neq 2,3$. Similarly, if $\mathrm{char}~R\neq 2,3$ the other direct summand is zero whenever $\lambda \ntrianglerighteq \lambda'$ or $\mu \ntriangleright \mu'$.  
%This implies that each of the direct summands has to be zero. If $\mathrm{char}~R\neq 2,3$, then $\Ext$ is zero, whenever $\lambda \ntriangleright \lambda'$ by [\cite{HN}, Theorem 4.2.1]. Or if $\mathrm{char}~R\neq 2$, then $\Hom$ becomes zero, whenever $\mu \ntrianglerighteq \mu'$. Similarly, the other direct summand is zero, whenever $\lambda \ntrianglerighteq \lambda'$ or $\mu \ntriangleright \mu'$. 
Hence, the proof follows.
\end{proof}

\begin{remark}
    We can determine the criteria for non-vanishing cohomology in the case of $\B_{r,t}(\delta)$ because the existence of an exact split pair ensures that the $\Ext$ between two cell modules of $\B_{r,t}(\delta)$ is isomorphic to the $\Ext$ between two cell modules of $R\Sg_{r-l,t-l}$. While a similar isomorphism holds for other algebras, the  criteria for non-vanishing cohomology for the case of wreath product algebras $A \wr \Sg_n$ remain unknown, which is needed to determine the non-vanishing cohomology for $D_n(A)$, and $\B_n^r(\bdel)$.
\end{remark}
Let $\mathrm{char}~R=p$. If $p \leq r $ (resp. $p \leq t$) then $R\Sg_r$ (resp. $R\Sg_t$) has infinite global dimension. This implies that $R\Sg_{r,t}$ has infinite global dimension if $p \leq \min \{r,t\}$.
% \begin{remark} \label{the condition of the direct product group has infinite global dimension}
% Let $\mathrm{char}~K=p$, a prime number. If $p \leq r $ (resp. $p \leq t$) then $K\Sg_r$ (resp.$K\Sg_t$) has infinite global dimension. This implies that $K\Sg_{r,t}$ has infinite global dimension if $p \leq \min \{r,t\}$.
% \end{remark}

\begin{corollary}\label{global dimension for walled brauer for split pair}
 $\B_{r,t}(\delta)$ has finite global dimension if and only if $p > \max\{r,t\}$, or $p=0$.
 % The statement still holds for the case $\delta=0$ and $r\neq t$.
\end{corollary}
\begin{proof}
If the $R$-algebra $\B_{r,t}(\delta)$ has finite global dimension, then $R\Sg_{r,t}$ has finite global dimension by [\cite{DK}, Corollary 1.6], and Theorem \ref{corner split quotient of walled Brauer algebras}. This gives $p > \max\{r,t\}$, or $p=0$. Conversely, if the given conditions hold, then $\B_{r,t}(\delta)$ is quasi-hereditary by [\cite{CVDM}, Corollary 2.8], which ensures $\B_{r,t}(\delta)$ has finite global dimension by [\cite{KX2}, Theorem 1.1].
% Assume that $\B_{r,t}(\delta)$ has a finite global dimension. Since the global dimension of $K\Sg_{r,t}$ is less than or equal to the global dimension of $\B_{r,t}(\delta)$ by [\cite{DK}, Corollary 1.6]. So, $K\Sg_{r,t}$ must have a finite global dimension. Therefore, we have $p > \max\{r,t\}$ by Remark \ref{the condition of the direct product group has infinite global dimension}. Conversely, if $p > \max\{r,t\}$ and $\delta \neq 0 $ or $\delta= 0$, and $r \neq t$, then $\B_{r,t}(\delta)$ is quasi-hereditary by [\cite{CVDM}, Corollary 2.8]. Since it is quasi-hereditary, it must have finite global dimension by [\cite{KX2}, Theorem 1.1]
\end{proof}
The statements in Theorem \ref{corner split quotient of walled Brauer algebras}, Corollary \ref{Hom between cell module for walled Brauer for split pair} and Corollary \ref{Ext between cell module for walled for split pair} holds true for the case $\delta=0$, and one of $r$ and $t$ greater than or equal to $2$. The Corollary \ref{global dimension for walled brauer for split pair} is still valid if $\delta=0$ with one of $r$ and $t$ greater than or equal to $2$, and $r \neq t$.

\subsection*{Acknowledgment}

The authors express their gratitude to Prof. Steffen Koenig for his time and valuable discussions. We would like to sincerely thank the referees for their thoughtful and constructive comments, which have significantly improved the quality of this paper. The first author's research is supported by Indian Institute of Science Education and Research Thiruvananthapuram PhD fellowship. The second author’s research was partially supported by IISER-Thiruvananthapuram, SERB-Power Grant SPG/2021/004200 and Prof. Steffen Koenig’s research grant. The second author also would like to acknowledge the Alexander Von Humboldt Foundation for their support.

\end{document}